\documentclass[11pt,a4paper]{amsart}
\usepackage{amsmath} 
\usepackage{lmodern}
\usepackage{amsmath, amssymb}
\usepackage[all]{xy}
\usepackage[utf8]{inputenc}
\usepackage[english]{babel}
\usepackage{graphicx}
\usepackage{mathtools} 
\usepackage{pifont}

\newtheorem{definition}{Definition}
\newtheorem{proposition}{Proposition}
\usepackage{hyperref}
\usepackage{fancyhdr}
\usepackage{amsthm}

\newtheorem{thm}{Theorem}[section]
\newtheorem{thmx}{Theorem}

\newtheorem{pro}{Proposition}[section]

\newtheorem{Def}{Definition}[section]

\newtheorem{lem}{Lemma}[section]

\begin{document}
\title{On the distribution of the periods of convex representations II}
\author{Abdelhamid Amroun}
\address{Laboratoire de Math\'ematiques d'Orsay, Univ. Paris-Sud, CNRS UMR 8628, Universit\'e Paris-Saclay, 91405 Orsay, France}
\email{abdelhamid.amroun@universite-paris-saclay.fr}

\begin{abstract} Let $\rho : \Gamma \longrightarrow G$ be a Zariski dense irreducible convex  representation of the hyperbolic group $\Gamma$, where $G$ is a  connected real semisimple algebraic Lie group. We establish a central limit type theorem for the periods of the representation $\rho$.
\end{abstract}
\maketitle
\tableofcontents
\section{Introduction and main results}
Asymptotic counting results are proved in \cite{samb1,samb3} for a class of convex representations $\rho : \Gamma \longrightarrow G$ of hyperbolic groups $\Gamma$ into semisimple algebraic Lie groups $G$. Among these results, Sambarino proved a prime orbit type theorem and an equidistribution result for the periods of the representation $\rho$.
By exploiting these results, we establish in this paper a central limit theorem for the periods of a class of $P$-convex representaions $\rho$ (see Definition \ref{def1}), where $P$ is any parabolic subgoup of $G$ (when $P$ is minimal, then the representation $\rho$ is hyperconvex, in the sens of \cite{samb3}). A more formal picture of the nature of the central limit type theorem we prove here, reads as follows. Let $Per_t(\rho)$ be the set of periods $\tau(\rho)$ of the representation $\rho$ (see below for precise definitions) such that $\tau(\rho)\leq t$. Then, for a certain class of observables $\varphi$, one can find constants $c_t^{\varphi}>0$, $L$ and $\sigma >0$  such that, for all $a, b\in \mathbb{R}$, 
\[
c_t^{\varphi}\#\{\tau(\rho)\in Per_t(\rho): \frac{\varphi(\tau(\rho))-Lt}{\sigma \sqrt{t}}\in [a,b]\}\longrightarrow \frac{1}{\sqrt{2\pi}}\int_a^b e^{\frac{x^2}{2}}dx,
\]as $t\rightarrow \infty$.
The constant $c_t^{\varphi}$ is the exponential growth rate of the set $Per_t(\rho)$ and satisfies, 
\[
c_t^{\varphi}\# Per_t(\rho) \longrightarrow 1, \ as \
t\rightarrow \infty.
\]
Let $\Gamma$ be a torsion-free discrete cocompact subgroup of the isometry group $Isom(\tilde{M})$ of a simply connected complete manifold $\tilde{M}$ of sectional curvature less or equal to $-1$. The group $\Gamma$ is hyperbolic and it's boundary $\partial \Gamma$ is naturally identifiefd with the geometric boundary $\partial \tilde{M}$ of $\tilde{M}$ \cite{Ghys}.
We consider $G$ a connected real semisimple algebraic Lie group with Lie algebra $\mathfrak{g}$. The group $G$ is then reductive, which means that it's unipotent radical $R_u(G)$ is trivial (examples are the linear groups $GL(n, \mathbb{R})$, $SL(n, \mathbb{R})$ or any torus). 

We are concerned by the irreducible representations $\rho : \Gamma \longrightarrow G$ of the hyperbolic groups $\Gamma$. More precisely, we are interested in representations $\rho$ which have equivariant (with some other additional conditions) maps $\psi : \partial \Gamma \rightarrow G/ P$, where $P$ is a proper parabolic subgroup of $G$ i.e. $P$ includes a Borel subgroup of $G$ or equivalently, the homogenuous space $G/ P$ is a complete variety \cite{milne}. When $P$ is a minimal parabolic subgroup of $G$, which is equivalent to say that $P$ is a Borel subgroup \cite{milne}, then, thanks to the Iwasawa decomposition, $\mathcal{F}:=G/ P$ is identified with the Furstenberg boundary $\partial_{\infty}X :=K/M$ of the Riemannian symetric space $X=G/K$ of $G$, where $K$ is a (maximal) isotopy subgroup of $G$ (see \cite{wa}). Recall that $X$ has a nonpositive curvature and $\partial_{\infty}X$ is the visual boundary of $X$ at infinity.
The parabolic subgroup $P$ is the stabilizer in $G$ of some point in $\partial_{\infty} X$. For example, when $G=PGL(d, \mathbb{R})$, then $\mathcal{F}$ is the space of complete flags of $\mathbb{R}^d$ with stabilizer, the projectivised of the group $M$ of diagonal matrices with $\pm 1$ on the diagonal.

It is usefull to describe parabolic subgroups via root systems. Let $\Delta$ be a system of simple roots of $G$ determined by $T$ and a Borel subgroup $B$ of $G$. There are $2^{|\Delta|}$ parabolic subgroups $P_{\theta}$ containing $B$  and corresponding to subsets $\theta$ in $\Delta$. The subgroup $P_{\Delta}$ corresponding to $\theta = \Delta$ is a minimal parabolic subgroup. 

Now, fix a maximal torus $T$ in $G$ i.e. $T$ is abelian and consists of semisimple elements and $T$ is not contained in an other torus with these properties. As for Borel and parabolic subgroups of $G$, all maximal tori are conjugated in $G$. Let $P$ be a parabolic subgroup of $G$ and $P^{-}$ it's opposite. There exists two root groups i.e. $1$-dimensional subgroups of $G$ normalized by $T$ such that $P=LU$ and $P^{-}=LU^{-}$ where $L=P\cap P^{-}$ is the unique Levi subgroup of $P$ (and $P^{-}$) containing $T$. The Levi factor $L$ is a connected reductive subgroup of $P$. It is clear that $L$ is the stabiliser in $G$ of the pair $(P, P^{-})$ so that the quotient space $G/L$ is the unique open $G$-orbit in the set $G/P \times G/P^{-}$. 
We set in the sequel, $\mathcal{F}_P = G/P$ and $\mathcal{F}_{P}^{2} = G/P^{-} \times G/P$.
 
Let $\mathfrak{a}_P$ be the Lie algebra of the center of the reductible group $L$ and set $\mathfrak{a} =\mathfrak{a}_{\Delta}$. 
Let $\mathfrak{J}_P : \mathfrak{a} \rightarrow \mathfrak{a}_P$ be the unique projection map $\mathfrak{J}_P : \mathfrak{a} \rightarrow \mathfrak{a}_P$ invariant under the subgroup $W_P$ of the Weyl group $W$, of elements $w$ which leave invariant $\mathfrak{a}_P$, i.e. such that $w(\mathfrak{a}_P)=\mathfrak{a}_P$. 
Thanks to the Iwasaw's decomposition $G=K \exp (\mathfrak{a}) N$, where $N$ is the unipotent radical of $P_{\Delta}$, there is a well defined map $\sigma=\sigma_{\Delta} : G \times \mathcal{F} \rightarrow \mathfrak{a}$, such that $gk=l \exp(\sigma (g, kM)) n$.
By a result of Quint \cite{Q}, the map $\mathfrak{J}_P\circ \sigma : G\times \mathcal{F} \rightarrow \mathfrak{a}_P$ factors through a map (called the Busemann-Iwasawa cocycle) $\sigma_P : G\times \mathcal{F}_P \rightarrow \mathfrak{a}_P$ which satisfy a cocycle equation, $\sigma_P(gh, x)=\sigma_P(g, hx)+\sigma_P(h, x)$ for every $g, h  \in G$ and $x\in \mathcal{F}_P$. 
Consequently, this leads to the following definition of a pair of dynamical cocycles $\beta_P : \Gamma \times \partial \Gamma \rightarrow \mathfrak{a}_P$, and $\beta_{P^{-}} : \Gamma \times \partial \Gamma \rightarrow \mathfrak{a}_{P^{-}}$ for the representation $\rho : \Gamma \rightarrow G$:    
\[
\beta_P (\gamma, x)=\sigma_P (\rho (\gamma), \xi^{+} (x)),  \
\beta_{P^{-}} (\gamma, x)=\sigma_P (\rho (\gamma), \xi^{-} (x)).
\]
By definition, the periods of the representation $\rho$ are the periods of the cocycle $\beta_P$. These are defined as the images $\beta_P (\gamma, \gamma^+)$, $\gamma \in \Gamma$, where $\gamma^+ \in \partial \Gamma$ is the attractive fixed point of $\gamma$. The periods of $\beta_P$ and $\beta_{P^{-}}$ are related by $i\beta_{P^{-}} (\gamma, \gamma^{-})=\beta_P (\gamma, \gamma^+)$, where $i :\mathfrak{a} \longrightarrow \mathfrak{a}$ is the opposition involution. By a result of Benoist \cite{Be}, for a Zariski dense\footnote{i.e. $\rho(\Gamma)$ is Zariski dense in $G$} $P$-convex representation $\rho : \Gamma \rightarrow G$, the closed cone containing the set $\mathcal{L}_P:=\{\beta_P (\gamma, \gamma^+) : \gamma \in \Gamma\}$ is convex and has nonempty interior. The dual cone $\mathcal{L}^{*}_{P}$ of linear forms $\varphi \in \mathfrak{a}^{*}$, such the restriction $\varphi_{|\mathcal{L}_P} \geq 0$, has a nonempty interior (see Sect. $3$). We introduce the set,
\[
[\Gamma]_t^{\varphi} = \{[\gamma]\in [\Gamma]: \varphi(\beta_P (\gamma, \gamma^+)) \leq t\}.
\]
We set in the sequel
$\mathcal{N}(a,b) :=\frac{1}{\sqrt{2\pi}} 
\int_{a}^{b} e^{-\frac{x^2}{2}}dx, \ a,b \in \mathbb{R}$.
The main results of this paper are the following.
\begin{thmx}[Theorem \ref{T4}]\label{A}
Let $\rho : \Gamma \rightarrow G$ be Zariski dense and $P$-convex irreducible representation, and fix $\varphi$ in the interior of $\mathcal{L}_P^{*}$, where $P$ is a parabolic subgroup of $G$.
There exist $L_P^{\varphi}$ and $\sigma_P^{\varphi} >0$ such that,
\[
h_{\varphi}te^{-h_{\varphi}t} 
\#  \{ [\gamma]\in [\Gamma]_t^{\varphi}: \frac{\varphi(\beta_P (\gamma, \gamma^+))-tL_P^{\varphi}}{\sigma_P^{\varphi} \sqrt{t}}\in [a,b]  \}
\rightarrow \mathcal{N}(a, b),  
\]as  $t\rightarrow \infty$.
\end{thmx}
Fix some weyl chamber $\mathfrak{a}^+$ in $\mathfrak{a}$ and let $a : G \longrightarrow \mathfrak{a}$ be the Cartan projection corresponding to the decomposition $G=K\exp(\mathfrak{a}^+) K$.
Set $a_P:= \mathfrak{J}_P  \circ a : G \longrightarrow \mathfrak{a}$.
\begin{thmx}[Theorem \ref{TT4}]\label{B}
Let $\rho : \Gamma \rightarrow G$ be Zariski dense and $P$-convex irreducible representation, and fix $\varphi$ in the interior of $\mathcal{L}_P^{*}$, where $P$ is a parabolic subgroup of $G$. Then,
\[
h_{\varphi}te^{-h_{\varphi}t} \#  \{[\gamma] \in [\Gamma]_t^{\varphi}: \frac{ \varphi(a_P((\rho\gamma)) )-tL_P^{\varphi}}{\sigma_P^{\varphi} \sqrt{t}}\in [a,b] \}
\rightarrow
\mathcal{N}(a, b),
\]
as $t\rightarrow \infty$, where $L_P^{\varphi}$ and $\sigma_P^{\varphi}$ are the constants of Theorem \ref{A}.
\end{thmx}
Let us explain the constants $L_P^{\varphi}$ and $\sigma_P^{\varphi}$.  
By the correspondance Theorem \ref{L} of Ledrappier, the cocycle $c_P^{\varphi}$ is cohomologus to a Hölder continuous function $F_P^{\varphi}$ with the same periods. Let $\mu_P^{\varphi}$ be the measure of maximal entropy of the Anosov flow $\psi^t :\Gamma\backslash \partial \Gamma^2 \times \mathbb{R} \circlearrowleft$ obtained as the reparametrisation, by $c_P^{\varphi}:=\varphi\circ \beta_P : \Gamma \times \partial \Gamma \longrightarrow \mathbb{R}$ (or similarily $F_P^{\varphi}$) of the geodesic flow on $\Gamma\backslash \partial \Gamma^2 \times \mathbb{R}$ (see Sect. $3.2$). Thus, applying Theorem \ref{sc} of Cantrell and Scharp (see Sect. $3.3$), we obtain that $L_P^{\varphi}=\int F_P^{\varphi} d\mu_P^{\varphi}$ and
\[
\sigma_P^{\varphi} =\lim_{t\rightarrow+\infty}\int \left(\int_0^{t}F_P^{\varphi} (\psi^t (x))dt-L_P^{\varphi}t \right )^2d\mu_P^{\varphi} (x).
\]
The normalization by the factor $h_{\varphi}te^{-h_{\varphi}t}$ comes from the prime orbit theorem (see Sect. 6, Theorem \ref{samba}).
\section{Convex representations}
We begin by this definition.
\begin{Def} [Sambarino \cite{samb1}]\label{def1}
We say that a representatio $\rho : \Gamma \longrightarrow G$ is $P$-convex,  where $P$ is a parabolic subgroup of $G$, if there exist two $\rho$-equivariant Hölder continuous maps $\xi^- : \partial \Gamma \longrightarrow \mathcal{F}_{P^{-}}$ and $\xi^+ : \partial \Gamma \longrightarrow \mathcal{F}_{P}$ such that, if $x\ne y$ are distinct points in $\partial \Gamma$, then the pair $(\xi^- (x), \xi^+ (y))$ belongs to $G/L$, i.e. $(\xi^- (x), \xi^+ (y))$ is a $G$-orbite in $\mathcal{F}_{P^{-}} \times \mathcal{F}_{P}$. The representation is said hyperconvex if $P$ is minimal, i.e. $P=P_{\Delta}$. In this special case we have $P=P^-$ and $\xi^- =\xi^+$.
\end{Def}

In the case of rational representations $\rho$, i.e. for $G=PGL(d, \mathbb{R})$, the parabolic group $P$ (resp $P^{-}$) is the subgroup of lower (resp upper) block triangular matricies, and $L$ is the subgroup of block diagonal matricies. Then $\mathcal{F}_P $ is the variety of the partial flags of $\mathbb{R}^d$. If in place of block triangular matrices we just take triangular ones, then $P$ is the (minimal) parabolic subgroup of $G$ of lower triangular matrices and $\mathcal{F}=G/P$ is the space of complete flags in $\mathbb{R}^d$. 

A rational representation $\rho : \Gamma \longrightarrow PGL(V)$ is said strictly convex (\cite{samb1} Definition 1.1), if there exists a $\rho$-equivariant Hölder continuous maps,
\[
\xi :\partial \Gamma \longrightarrow \mathbb{P}(V), \ and \ \eta:  \partial \Gamma \longrightarrow \mathbb{P}(V^*),
\]
such that $\xi (x)\oplus \eta (y) =V$ for all $x\ne y$ in $\partial \Gamma$.

Consider a Zariski dense hyperconvex representation $\rho : \Gamma \longrightarrow G$ (i.e. $P_{\Delta}$-convex) with equivariant map $\xi: \partial \Gamma \longrightarrow \mathcal{F}_{P_{\Delta}}$. 
By Chevalley's theorem \cite{hum}, there is a finite dimensional irreducible representation $\tilde{\rho} : G \longrightarrow PGL(V)$ and a line $L \subset V$ such that $P_\Delta$ fixes $L$, i.e. such that $\tilde{\rho} (g)L=L$ for all $g\in P_{\Delta}$. In other words, $P_{\Delta}$ is the isotropy subgroup of $G$ of the point $[L] \in \mathbb{P}(V)$. This defines an equivariant map $\zeta :\mathcal{F}_{P_{\Delta}} \longrightarrow \mathbb{P}(V)$, $\zeta (gP_\Delta)=\tilde{\rho} (g)[L]$. Considering the dual representation $\tilde{\rho}^* : G \longrightarrow \mathbb{P}(V^*)$, defined by $\tilde{\rho}^* (g)\phi = \phi\circ \rho (g^{-1})$ for all $g\in G, \phi \in \mathbb{P}(V^*)$, one obatains an equivariant map $\zeta^* :\mathcal{F}_{P_{\Delta}} \longrightarrow \mathbb{P}(V^*)$. 
The maps 
\[
\partial \Gamma \stackrel{ \text{$\xi$} }{\longrightarrow} \mathcal{F}_{P_{\Delta}} \stackrel{ \text{$\zeta$} }{\longrightarrow} \mathbb{P}(V),\ and \
\partial \Gamma \stackrel{ \text{$\xi$} }{\longrightarrow} \mathcal{F}_{P_{\Delta}} \stackrel{ \text{$\zeta^*$} }{\longrightarrow} \mathbb{P}(V^*),
\]
are then equivariant. The equivariant boundary maps $\zeta \circ \xi$ and $\zeta^* \circ \xi$ are in fact all Hölder continuous.
Since $\rho$ is Zariski dense, which means that $\rho (\Gamma)$ is Zariski dense in $G$, the representation $\tilde{\rho}\circ \rho$ is irreducible. The representation $\tilde{\rho}\circ \rho : \Gamma \longrightarrow PGL(V)$ is then  strictly convex with equivariant maps $\zeta \circ \xi$ and $\zeta^* \circ \xi$.

A rational representation $\tilde{\rho} : G \longrightarrow PGL(V)$ is said proximal if the eigenspace associated to the highest weight of $\rho$ is one dimensional. By a result of Tits \cite{tits}, given a parabolic subgroup $P$ of $G$, then $\rho (P)$ stabilizes a unique line in $V$. This defines a unique equivariant map $\mathcal{F}_P \longrightarrow \mathbb{P}(V)$. The dual representation $\tilde{\rho}^* : G \longrightarrow \mathbb{P}(V^*)$, is also irreducible and proximal. This defines again uniquely an equivariant map $\mathcal{F}_P \longrightarrow \mathbb{P}(V^*)$.
Consequently, if $\rho : \Gamma \longrightarrow G$ is $P$-convex Zariski dense irreducible representation, and $\tilde{\rho} : G \longrightarrow PGL(V)$ proximal, then $\tilde{\rho} \circ\rho$ is striclty convex.

More generally, if $\rho : \Gamma \longrightarrow G$ is Zariski dense and $P$-convex representation and $\tilde{\rho} : G \longrightarrow PGL(V)$ is a finite dimensional irreducible proximal representation, then the composition $\tilde{\rho}\circ \rho : \Gamma \longrightarrow PGL(V)$ is a strictly convex representation \cite{samb3}. The existence of $\tilde{\rho}$ results from Tit's theorem \cite{tits}.

They are many important exemples of strictly convex representations (see \cite{samb1}). Among these we find the class of Hitchin representations of surface groups. Moreover, the composition of a Zariski dense Hitchin representation $\Gamma \longrightarrow PGL(d, \mathbb{R})$
with some irreducible represenation  $PGL(d, \mathbb{R}) \longrightarrow PGL(k, \mathbb{R})$ is stricly convex.

\section{Cocycles, Periods and Gromov products}
For some general facts on Hölder cocycles related to the action of $\Gamma$ on $\partial \Gamma$, see Sect. $5$. For the purpose of this section, we concentrate on those ones related to the representation $\rho$.
\subsection{The Busemann-Iwasawa cocycle} Let $\mathfrak{a}_P$ be the Lie algebra of the center of the reductible group $L$ and set $\mathfrak{a} =\mathfrak{a}_{\Delta}$. The  subalgebra $\mathfrak{a}$ is a maximal abelian subspace contained in the $(-1)$-eigenspace of the Cartan involution on $\mathfrak{g}$. We will also need the unique projection map $\mathfrak{J}_P : \mathfrak{a} \rightarrow \mathfrak{a}_P$ invariant under the subgroup $W_P$ of the Weyl group $W$ (of $G$ relative to $T$) of elements $w$ which leave invariant $\mathfrak{a}_P$, i.e. such that $w(\mathfrak{a}_P)=\mathfrak{a}_P$.

As for the rational representations, one can associate to the (convex) representation $\rho : \Gamma \rightarrow G$ a cocycle map as follows \cite{Q}\cite{samb3}. 
Thanks to the Iwasaw's decomposition $G=K \exp (\mathfrak{a}) N$, where $N$ is the unipotent radical of $P_{\Delta}$, there is a well defined map $\sigma=\sigma_{\Delta} : G \times \mathcal{F} \rightarrow \mathfrak{a}$, such that $gk=l \exp\sigma (g, kM) n$.
Then, by a result of Quint \cite{Q}, the map $\mathfrak{J}_P\circ \sigma : G\times \mathcal{F} \rightarrow \mathfrak{a}_P$ factors through a map $\sigma_P : G\times \mathcal{F}_P \rightarrow \mathfrak{a}_P$ which satisfy a cocycle equation, $\sigma_P(gh, x)=\sigma_P(g, hx)+\sigma_P(h, x)$ for every $g, h  \in G$ and $x\in \mathcal{F}_P$. Consequently, this leads to the following definition of a pair of dynamical cocycles $\beta_P : \Gamma \times \partial \Gamma \rightarrow \mathfrak{a}_P$, and $\beta_{P^{-}} : \Gamma \times \partial \Gamma \rightarrow \mathfrak{a}_{P^{-}}$ for the representation $\rho : \Gamma \rightarrow G$:    
\[
\beta_P (\gamma, x)=\sigma_P (\rho (\gamma), \xi^{+} (x)),  \
\beta_{P^{-}} (\gamma, x)=\sigma_P (\rho (\gamma), \xi^{-} (x)).
\]
\subsection{Periods} 
Fix some weyl chamber $\mathfrak{a}^+$ in $\mathfrak{a}$ and let $a : G \longrightarrow \mathfrak{a}$ the Cartan projection corresponding to the decomposition $G=K\exp(\mathfrak{a}^+) K$.
Let $\lambda : G\rightarrow \mathfrak{a}^+$ be the Jordan projection, and
define $\lambda_P = \mathfrak{J}_P  \circ \lambda : G \rightarrow \mathfrak{a}_P$.

By definition, the periods of the cocycle $\beta_P$ are the images of the function $\gamma \in \Gamma \rightarrow \beta_P (\gamma, \gamma^+)$.
The periods of $\beta_P$ and $\beta_{P^{-}}$ are related by $i\beta_{P^{-}} (\gamma, \gamma^{-})=\beta_P (\gamma, \gamma^+)$, where $i :\mathfrak{a} \longrightarrow \mathfrak{a}$ is the opposition involution.
We have the following result.
\begin{pro}[\cite{Q} \cite{samb3}]\label{Prop2}
Let $\rho : \Gamma \rightarrow G$ be Zariski dense and $P$-convex representation, where $P$ is a parabolic subgroup of $G$. We have,
\[
\beta_P (\gamma, \gamma^+) =\lambda_P(\rho\gamma), \ \forall \gamma \in \Gamma.
\]
\end{pro}
By a result of Benoist \cite{Be}, since $\rho (\Gamma)$ is Zariski dense in $G$, then the closed cone containing the set $\mathcal{L}_{\rho}:=\{\lambda (\rho(\gamma)): \gamma \in \Gamma\}$ is convex and has nonempty interior.

Let $\mathcal{L}_P$ be the closed cone generated by the periods $\lambda_P(\rho\gamma)$, $\gamma \in \Gamma$. We have $\mathcal{L}_P=\mathfrak{J}_P (\mathcal{L}_{\rho})$, then $\mathcal{L}_P$ is a convex and closed.
Consider its dual cone $\mathcal{L}_P^{*}$ of linear forms $\varphi \in \mathfrak{a}_{P}^{*}$ such that the restriction of $\varphi$ to $\mathcal{L}_P$ is nonnegative, $\varphi_{|\mathcal{L}_P} \geq 0$. Since $\mathcal{L}_P \subset \mathfrak{J}_P  (\mathfrak{a}^+)$, the cone $\mathcal{L}_P$ does not contain a line. It follows then that the dual cone $\mathcal{L}_P^{*}$ has a nonempty interior. We can summerize this as follows.
\begin{lem}\label{PC}
The cones $\mathcal{L}_P$ and $\mathcal{L}_P^{*}$ are proper, i.e. convex closed cones which do not contain a line and have nonempty interior.
\end{lem}

\subsection{Gromov product on $ \mathcal{F}_{P^{-}}\times \mathcal{F}_P$}
When we want to assess how a couple of points $(x,y)\in \mathcal{F}_{P^{-}}\times \mathcal{F}_P$ are far a part, one can use the notion of proximality (see Sect. $4$). For this,
we introduce the Gromov product which is a $\mathfrak{a}_P$-valued map $\mathcal{G}_P : \mathcal{F}_{P^{-}}\times \mathcal{F}_P \longrightarrow \mathfrak{a}_P$ defined as follows.

Set $P=P_{\{\alpha_1, \cdots, \alpha_k\}}$, where the $\alpha_i$'s are the system of roots defining the parabolic subgroup $P$.
By the Tits' representation theorem \cite{tits}, for each $\alpha \in \{\alpha_1, \cdots, \alpha_k\}$ there exists a finite dimensional proximal representation $\Lambda_{\alpha} : G \longrightarrow PGL(V_\alpha)$ such the highest weight $\chi_\alpha$ of $\Lambda_{\alpha}$ is an integer multiple of the fundamental $\omega_\alpha$. We get in particular that the family $\{\chi_{\alpha_1}, \cdots, \chi_{\alpha_k}\}$ is a basis of $\mathfrak{a}_{P}^{*}$.

On the other hand, since $\Lambda_{\alpha}$ is proximal, we also get a continuous equivariant map $\xi_\alpha : \mathcal{F}_P \longrightarrow \mathbb{P}(V_\alpha)$. Moreover, considering the dual representation $\Lambda_{\alpha}^{*} : G \longrightarrow PGL(V_{\alpha}^{*})$, one obtains a continuous equivariant map $\eta_\alpha : \mathcal{F}_{P^{-}} \longrightarrow \mathbb{P}(V_{\alpha}^{*})$ such that, for any pair of points $(x,y)\in \mathcal{F}_{P^{-}}\times \mathcal{F}_P$, in general position i.e. $x\ne y$, the line $\xi_\alpha (y)$ is not contained in $\eta_\alpha (x)$ (then $\xi_\alpha (y)\oplus \eta_\alpha (x) =V_\alpha$).

Fix an Euclidean norm $\|\ \|_\alpha$ ("a good norm" in the sens of \cite{Q}), on the finite dimensional vector space $V_\alpha$. Given $(x,y)$ as above, $\mathcal{G}_P(x,y)$ is entirely determined by $\chi_{\alpha} (\mathcal{G}_P(x,y))$, for $\alpha \in \{\alpha_1, \cdots, \alpha_k\}$. Following \cite{samb1} we set,
\[
\chi_{\alpha} (\mathcal{G}_P(x,y)) =\log \frac{|\theta (v)|}{\|\theta\|_\alpha\|v\|_\alpha},
\]where $v\in \xi_\alpha (y)$ and $\theta \in \eta_\alpha (x)$.

The following lemma is the Lemma 7.9 in \cite{samb1}.
\begin{lem}\label{CG}
For all $g\in G$ and $(x,y)\in \mathcal{F}_{P^{-}}\times \mathcal{F}_P$ in general position we have,
\[
\mathcal{G}_P(gx, gy)- \mathcal{G}_P(x, y)=
-i \sigma_{P^-}(g, x)- \sigma_P(g, y).
\]
\end{lem}
\begin{proof}
For any $(x,y)\in \mathcal{F}_{P^{-}}\times \mathcal{F}_P$ and $g\in G$ we have,
\[
\chi_\alpha (\mathcal{G}_P(gx, gy))=\log \frac{|\theta (v)|}{\|\theta\|_\alpha\|v\|_\alpha},
\]where $v\in \xi_\alpha (gy)$ and $\theta \in \eta_\alpha (gx)$, which is equivalent to $\Lambda_\alpha (g^{-1})v \in \xi_\alpha (y)$ and $\theta \circ \Lambda_\alpha (g) \in \eta_\alpha (x)$.
We deduce from this that,
\[\chi_\alpha (\mathcal{G}_P(x, y))
=\log\frac{|\theta \circ \Lambda_\alpha (g) (\Lambda_\alpha (g^{-1})v)|}{\| \theta \circ \Lambda_\alpha (g)\|_\alpha\|\Lambda_\alpha (g^{-1})v\|_\alpha}
\]
\[=-\log \frac{\|\theta \circ \Lambda_\alpha (g)\|_{\alpha}}{\|\theta\|_{\alpha}}-\log\frac{\|\Lambda_\alpha (g^{-1})v\|_{\alpha}}{\|v\|_{\alpha}}+\log \frac{|\theta (v)|}{\|\theta\|_\alpha\|v\|_\alpha}
\]
Thus,
\[
\chi_\alpha (\mathcal{G}_P(gx, gy))- \chi_\alpha (\mathcal{G}_P(x, y))=
\log \frac{\|\theta \circ \Lambda_\alpha (g)\|_{\alpha}}{\|\theta\|_{\alpha}}
+\log\frac{\|\Lambda_\alpha (g^{-1})v\|_{\alpha})}{\|v\|_{\alpha}}.
\]
Now applying Lemma $6.4$ in \cite{Q} for the representation $\Lambda_{\alpha} : G \longrightarrow PGL(V_\alpha)$ we find that,
\[
\log\frac{\|\Lambda_\alpha (g^{-1})v\|_{\alpha}}{\|v\|_{\alpha}}=\chi_{\alpha} (\sigma_P(g^{-1}, gy))=
-\chi_{\alpha} (\sigma_P(g, y)).
\]
On the other hand, considering the dual representation $\Lambda_{\alpha}^{*} : G \longrightarrow PGL(V_{\alpha}^{*})$, we get using the same lemma,
\[
\log\frac{\|\Lambda_\alpha (g^{-1})\theta\|_{\alpha}}{\|\theta\|_{\alpha}}=\chi_{\alpha} \circ i (\sigma_{P^-}(g^{-1}, gx))=
-\chi_{\alpha} \circ i (\sigma_{P^-}(g, x))
\]
Finally we get to the desired relation,
\[
\mathcal{G}_P(gx, gy)- \mathcal{G}_P(x, y)=
-(i \sigma_{P^-}(g, x)+ \sigma_P(g, y)).
\]
\end{proof}
\subsection{Gromov product on $\partial \Gamma \times \partial \Gamma$}
Define the $\mathfrak{a}_P$-valued Gromov product $[ \ ,\ ]_P: \partial\Gamma \times \partial\Gamma \longrightarrow \mathfrak{a}_P$ by, 
\[
[x,y]_P= \mathcal{G}_P \ (\xi^-(x),\xi^+(y)).
\]
We have, $[\gamma x, \gamma y]_P=\mathcal{G}_P \ (\rho(\gamma)\xi^-(x),\rho(\gamma)\xi^+(y))$
As a consequence we have for all $(x,y)\in \partial\Gamma \times \partial\Gamma$ and $\gamma \in \Gamma$,
\[
[\gamma x, \gamma y]_P-[x,y]_P=-(i\sigma_{P^-}(\rho(\gamma), \xi^-(x))+ \sigma_{P}(\rho(\gamma), \xi^+(y))).
\]
Set $\overline{\beta}_{P}=i\beta_{iP}=i\beta_{P^-}$. Then the above equality writes,
\[
[\gamma x, \gamma y]_P-[x,y]_P=-(\overline{\beta}_{P}(\gamma, x)+\beta_P(\gamma,y)).
\]
The pair $\{\overline{\beta}_{P}, \beta_P\}$ is a pair of dual (in the sens of \cite{Led}) cocycles since the periods of $\overline{\beta}_{P}$ are given by $\overline{\beta}_{P}(\gamma, \gamma^+)=\lambda_P(\rho(\gamma^{-1}))$.
This means that $[ \ ,\ ]_P$ defines a Gromov product for the pair $\{\overline{\beta}_{P}, \beta_P\}$.

Consider $\varphi \in \mathcal{L}_P^{*}$ and define the real valued function $[\cdot, \cdot]_\varphi :\partial\Gamma \times \partial\Gamma \longrightarrow \mathbb{R}$ by,
\[
[\gamma, \gamma']_\varphi =\varphi \circ \mathcal{G}_P \ (\xi^-(\gamma),\xi^+(\gamma'))
\]
The followin lemma is the Lemma 7.10 in \cite{samb1}.
\begin{lem}\label{PG}
Let $\rho : \Gamma \rightarrow G$ be Zariski dense and $P$-convex representation and consider $\varphi$ in the interior of the dual cone $\mathcal{L}_P^{*}$. Then the function $[\cdot, \cdot]_\varphi$ is a Gromov product for the pair of dual cocycles $\beta_\varphi, \overline{\beta}_\varphi$.
\end{lem}	
\section{Proximality}
We begin by recalling the notion of proximality. Following \cite{samb3}, an element $g\in G$ is said to be proximal on $\mathcal{F}_P$ (or simply $P$-proximal), if it has an attracting fixed point $g_{+}^P \in \mathcal{F}_P$. The repelling point of $g$ on $\mathcal{F}_{P^{-}}$, is the fixed point $g_{-}^P$ of $g$ on $\mathcal{F}_{P^{-}}$. It satisfies $g^n x \longrightarrow g_{+}^P$, as $n\rightarrow \infty$,  for every point $x\in \mathcal{F}_P$ such that $( g_{-}^P, x)$ belongs to $\mathcal{F}_{P}^{2}$.
\begin{lem}\label{UD}
For all $r>0$ and $\epsilon>0$ we have,
\[
h_{\varphi}te^{-h_\varphi t}\#\{[\gamma]\in [\Gamma]_t^\varphi:  \ \rho (\gamma) \ is \ (r,\epsilon)-proximal\}
 \longrightarrow 1, \ as \ t\rightarrow \infty.
\]
\end{lem}
For Zariski-dense $P$-convex representations, we have the following result (see \cite{samb2} Corollary 3.12).
\begin{pro}[Sambarino \cite{samb2}]\label{sa}
Let $\rho : \Gamma \rightarrow G$ be Zariski dense and $P$-convex representation, where $P$ is a parabolic subgroup of $G$. Then, for every $\gamma \in \Gamma$, $\rho (\gamma)$ is $P$-proximal, $\xi^{+}(\gamma_+)$ is its attracting point and $\xi^{-}(\gamma_-)$ is the repelling point.
\end{pro}

The notion of $(r, \epsilon)$ proximality of Benoist extends to this situation as follows.
\begin{definition}[Benoist \cite{Be}]\label{bee}
We say that $g\in G$ is $(r, \epsilon)$-proximal on $\mathcal{F}_P$ for some $r>0$ and $\delta >0$ if it is proximal, $\exp \|\mathcal{G}_P (g_{-}^P , g_{+}^P )\|_{\mathfrak{a}_P} >r$ and the complement of an $\epsilon$-neighborhood of $g_{-}^P$ is sent by $g$ to an $\epsilon$-neighborhood of $g_{+}^P$.
\end{definition}
We have the following result on the $\gamma$'s for which $\rho (\gamma)$ is not proximal.
\begin{pro}[Sambarino \cite{samb2}]\label{sam}
Let $\rho : \Gamma \rightarrow G$ be Zariski dense and $P$-convex representation, where $P$ is a parabolic subgroup of $G$. Fix some $r>0$ and $\epsilon >0$. Then the following set is finite:
\[
\{\gamma\in \Gamma : \exp (\|\mathcal{G}(\eta (\gamma_{-}), \xi (\gamma_{+}))\|_{\mathfrak{a}_P})>r\ and \ \rho (\gamma)\ is\ not\ (r, \epsilon)-proximal\}.
\]
\end{pro}
\subsection{Proof of Lemma \ref{UD}}
\begin{proof}
First of all, observe that each conjugacy class $[\cdot]\in [\Gamma]$ has a representative $\gamma$ whose end fixed points $\gamma_{-}$ and $\gamma_{+}$ are far appart (with respect to the Gromov distance $d_G$ on the boundary $\partial \Gamma$) by a positive constant $\kappa$ independant from the class $c$ (since $\Gamma$ acts cocompactly i.e. with a compact fundamental domain). 
Then, since $\xi^-$ and $\xi^+$ are uniformely continuous, by the continuity of the positive function $ \exp [x, y]_\varphi$ on the compact set $\{(x,y)\in\partial \Gamma \times \partial\Gamma: d_{G}(x,y)\geq \kappa\}$,  
there exists a positive constant $r>0$ such that every conjugacy class $[\cdot]\in [\Gamma]$ can be represented by some $\gamma \in \Gamma$ with $\exp [\gamma_{-}, \gamma_{+}]_\varphi > r$.
Then we have for all $n>0$, 
\begin{equation}\label{E2}
\# [\Gamma]_t^\varphi
=\#\{[\gamma]\in [\Gamma]_t^\varphi  : \exp [\gamma_{-}, \gamma_{+}]_{\varphi} >r\}.
\end{equation}
By Proposition \ref{sam}, all $\gamma$'s in $(1)$ 
(except a finit number depending only on $r$ and $\epsilon$) are $(r, \epsilon)$-proximal. Thus, using the prime orbit theorem (see Sect. 6 Theorem \ref{samba}) applied to the cocycle $c_P^\varphi =\varphi \circ \beta_P$, we obtain
\[
h_{\varphi}te^{-h_\varphi t}\#\{[\gamma]\in [\Gamma]_t^\varphi: \rho (\gamma)\ is\ (r, \epsilon)-proximal\}\rightarrow 1,
\]as $t$ goes to infinity.
This proves the lemma.
\end{proof}
\section{Hölder cocycles and Ledrappier's correspondance}
\subsection{Hölder continuous cocycles}
The main reference in this section is the paper \cite{Led} by F. Ledrappier.
\begin{Def}[Ledrappier \cite{Led}]\label{def2}
A cocycle over $\partial \Gamma$, is a real valued function
$c:  \Gamma \times \partial \Gamma\rightarrow \mathbb{R}$,
such that, for all $\gamma_1, \gamma_2 \in \Gamma$ and $\xi \in \partial \Gamma$, 
\[
c(\gamma_1 \gamma_2, \xi)=c(\gamma_2, \gamma_1\cdot \xi)+c(\gamma_1, \xi).
\]
If, for all $\gamma \in \Gamma$, the map $\xi \rightarrow c(\gamma, \xi)$ is Hölder continuous on $\partial \Gamma$, we say that the cocycle $c$ is Hölder. The cocycle $c$ is positive if $c(\gamma, \xi)>0$, for all $(\gamma, \xi)\in   \Gamma \times \partial \Gamma$.
\end{Def}

Two Hölder cocycles are cohomologically equivalent if they differ by a Hölder continuous function $U: \partial M\rightarrow \mathbb{R}$ such that,
\[
c_1(\gamma, \xi)+c_2(\gamma, \xi)=U(\gamma \cdot \xi)-U(\xi).
\]
Given $\gamma \in \Gamma$, recall that $\gamma_+$ it's attractive fixed point in $\partial \Gamma$.
The numbers $c(\gamma, \gamma_{+})$ depend only on the conjugacy class $[\gamma]\in \Gamma$ and on the cohomological class of $c$ \cite{Led}. We call $c(\gamma, \gamma_{+})$ the periods of $c$. 
Recall the following important result by Ledrappier (see \cite{Led} p104-105).
\begin{thm}[Ledrappier \cite{Led}]\label{L}
The Liv$\check{s}$ic cohomological classes of $\Gamma$-invariant 
$\mathcal{C}^2$-functions $F: T^1M \rightarrow \mathbb{R}$
are in one-to-one correspondance with the cohomological classes of Hölder cocycles $c:  \Gamma \times \partial \Gamma\rightarrow \mathbb{R}$.
Moreover, the classes in correspondance have the same periods, i.e. $c(\gamma, \gamma^{+})=\int_{p}^{\gamma p}F$, the integral of $F$ over the geodesic segment $[p, \gamma p]$ (for any $p\in M$).
\end{thm}
\subsection{A central limit theorem for Hölder cocycles}
For a cocycle $c: \Gamma \times \partial \Gamma \rightarrow \mathbb{R}$ with positive periods, we set
\[
[\Gamma]_t^c=\# \{ [\gamma]\in [\Gamma]: c(\gamma, \gamma_{+}) \leq t \}.
\]
\begin{thm}\label{T1}
Let $c: \Gamma \times \partial \Gamma \rightarrow \mathbb{R}$ be a Hölder continuous cocycle with positive periods and finite exponential growth rate $h_c$,
\[
h_c := \limsup_{t \longrightarrow \infty}
\frac{ \log \#[\Gamma]_t^c}{t} 
<\infty.
\]
There exist two constants $L_c>0$ and $\sigma_c>0$ such that for any $a, b\in \mathbb{R}$ with $a<b$ we have
\[
h_cte^{-h_c t} \# \{[\gamma]\in [\Gamma]_t^c : \frac{c(\gamma, \gamma_{+}) -L_ct}{\sigma_c \sqrt{t}}\in [a,b]\}
\rightarrow
 \mathcal{N}(a, b), \ as \ t\rightarrow \infty.
\]
\end{thm}
The constants $L_c$ and $\sigma_c$ depend only on the cocycle $c$.
By the correspondance Theorem \ref{L} of Ledrappier, the cocycle $c$ is cohomologus to a Hölder continuous function $F_c$ with the same periods. Let $\mu_c$ be the measure of maximal entropy of the Anosov flow $\psi^t :\Gamma\backslash \partial \Gamma^2 \times \mathbb{R} \circlearrowleft$ obtained as the reparametrisation by $c$ (or $F$) of the geodesic flow on $\Gamma\backslash \partial \Gamma^2 \times \mathbb{R}$ (see Theorem \ref{samb} Sect. $5.3$). Thus, applying Theorem \ref{sc} (see Sect. $5.3$), we obtain that $L_c=\int F_c d\mu_c$ and
\[
\sigma_c =\lim_{t\rightarrow+\infty}\int \left(\int_0^{t}F_c(\psi^t (x))dt-L_ct \right )^2d\mu_c (x).
\]
\subsection{Reparametrization of a Hölder continuous cocycle}
Let $X$ be a compact metric space and $\varphi^{t}: X \rightarrow X$ a continuous flow without fixed points. We consider in this paper, Hölder continuous cocycles $c$ over the flow $\varphi^{t}$.
A cocycle $c$ over the flow $\varphi^{t}$, is a function $c :X\times \mathbb{R} \rightarrow \mathbb{R}$ which satisfies the conditions:
\begin{itemize}
\item For all $x\in X$ and $s, t \in \mathbb{R}$,
\begin{equation}
c(x, s+t)=c(\varphi_t (x), s) + c(x,t) \label{eq: E} \tag{*},\ and
\end{equation}
\item For all $t\in \mathbb{R}$, the function $x\rightarrow c(x,t)$ is Hölder continuous (the exponent being independent from $t$).
\end{itemize}

Fundamental examples of such cocycles are given by a Hölder continuous function $F: X\times \mathbb{R \rightarrow} \mathbb{R}$ by setting,
$c_F (x,t)=\int_{0}^{t}F( \varphi_s (x) )ds$ for $t\geq 0$ and, 
$c_F (x,t)= -c_F (\varphi_t (x),-t)$, for $t<0$.

A cocycle $c$ is positive if for all $x\in X$, $c(x,t)>0$ for all $t\in \mathbb{R}$. In this case, for each $x\in X$, the function $t\rightarrow c(x,t)$ is increasing and in fact, it defines a homeomorphism of $\mathbb{R}$. It makes sens to consider the inverse cocyle $\hat{c}$:
\begin{equation}
\hat{c} (x, c(x,t))=c(x, \hat{c}(x,t))=t, \forall \ x\in X \label{eq: EE} \tag{**}.
\end{equation}
\begin{Def}\label{D1}
The reparametrization of the flow $\varphi$ by the positive cocycle $c$, is the flow $\psi$ defined for all $x\in X$ by $\psi^{t}(x)= \varphi^{\hat{c}(x,t)}(x)$. The flow $\psi$ is indeed Hölder by \eqref{eq: E} and \eqref{eq: EE}. Furthermore, both share the same periodic orbits; if $p(\theta)$ is the period of the periodic $\varphi$-orbit $\theta$, then $c(x, p(\theta))$ is it's period as a periodic $\psi$-orbit for all $x\in \theta$. 
\end{Def}
We end this part by the following reparametrization theorem of Sambarino \cite{samb1}.
\begin{thm}[Sambarino \cite{samb1}]\label{samb}
Let $c$ be a Hölder cocycle with positive periods such that $h_c$ is finite and positive. Then the following holds.
\begin{enumerate}
\item The action of $\Gamma$ on $\partial^2\Gamma \times \mathbb{R}$,
\[
\gamma(x,y,s)=(\gamma x, \gamma y, s-c(\gamma, y)),
\]
is proper and cocompact. The translation flow $\psi^t: \Gamma \backslash\partial^2\Gamma \times \mathbb{R} \circlearrowleft$,
\[
\psi^t \Gamma(x,y,s)=\Gamma(x, y, s-t)),
\]
is conjugated to a Hölder reparametrization of the geodesic flow $\mathcal{G}^t :\Gamma \backslash T^1\tilde{M}\circlearrowleft$.
\end{enumerate}
\end{thm}
By Ledrappier's Theorem \ref{L}, one can set $c=c_F$ for some Hölder continuous function $F$.
The reparametrization in the above theorem means that there exists a $\Gamma$-equivariant homeomorphism $E : T^1\tilde{M} \longrightarrow \partial^2\Gamma \times \mathbb{R}$ such that, for all $x=(p, v)\in T^1\tilde{M}$
\[
E(\mathcal{G}^{t} (p,v))=\psi^{c(x,t)}(E(p,v)),
\]where $c(x,t)=\int_0^t F(\mathcal{G}^s(x))ds$. It was proved in \cite{samb1} that $E$ is given by
\[
E(p,v)=(v_{-}, v_{+}, B_{v_{+}}^F)(p,o),
\]
where $o\in \tilde{M}$ is a fixed point (a base point), $v_{-}$ and $v_{+}$ are the end points in $\partial \tilde{M}$ of the geodesic with origine at $(p, v)\in T^1\tilde{M}$, and $\partial \tilde{M}\ni\xi \rightarrow B_{\xi}^F$ is the Busemann function based on the function $F$ (see \cite{samb1}).

Setting $F=1$ leads to the Hopf parametrization of the unit tangent bundle (based on the Busemann cocycle $B_{v_{+}}$),
\[
T^1\tilde{M}\ni(p,v)\longrightarrow (v_{-}, v_{+}, B_{v_{+}}(p,o)) \in \partial^2\Gamma \times \mathbb{R}.
\]
The geodesic flow of $T^1\tilde{M}$ is, by the way, a translation flow on $\partial^2\Gamma \times \mathbb{R}$,
\[
\mathcal{G}^t (v_{-}, v_{+}, B_{v_{+}}(p,o))=(v_{-}, v_{+}, B_{v_{+}}(p,o)+t).
\] 
\subsection{A central limit theorem for hyperbolic flows}
The geodesic flow of a negatively curved compact manifold does not admit a cross section. This a consequence of the Preissman theorem \cite{pr} (see \cite{samb3}).
Recall that a cross section for a flow $\phi^t : X\circlearrowleft$ is a closed subset $K$ of $X$ such that the function $K\times \mathbb{R} \ni (x,t)\rightarrow \phi^t (x)$ is a surjective local homeomorphism. This means essentially that the flow $\phi^t$ is not a suspension of a continuous flow. Moreover, this property is invariant under reparametrization \cite{samb1}. More precisely, the flow $\phi^t$ admits a cross section if and only if the same is true for any reparametrization $\psi^t$ of $\phi^t$. Consequently, the flow $\psi^t$ of Theorem \ref{samb} does not admit a cross section. This is equivalent to say that the subgroup of $\mathbb{R}$ generated by the periods of $c=c_F$ (i.e. the subgroup generated by $\{\int_{\tau}F: \tau \ periodic\}$) is dense \cite{samb3}. We can thus apply the following result of Cantrell and Scharp \cite{cs} (see also \cite{cs}, Remark 6.4 ) to the periodic orbits $\tau$ of the flow $\psi^t$.
\begin{thm}[Cantrell-Scharp \cite{cs}]\label{sc}
Suppose that $\psi^t : \Lambda \longrightarrow \Lambda$ is either a transitive Anosov flow with stable and unstable foliations which are not jointly integrable or a hyperbolic flow satisfying the approximability condition. Let $f:\Lambda \rightarrow\mathbb{R}$ be a Hölder continuous function that is not a coboundary. Then, there exists two constants $L$ and $\sigma_f >0$ such that, for all $a,b\mathbb{R}$,
\[
\frac{\#\{\tau \ periodic,\ p(\tau)\leq t: 
\frac{\int_{\tau}f-Lt}{\sigma_f\sqrt{t}}\in [a, b]\}}{\#\{\tau \ periodic,\ p(\tau)\leq t\}}\longrightarrow \mathcal{N}(a, b), 
\]as $t\rightarrow \infty$.
\end{thm}
In the theorem, $p(\tau)$ is the least period of the periodic orbit $\tau$. Let $\mu$ be the measure of maximal entropy of the flow $\psi^t$. The constants $L$ and $\sigma_f$ are given respectively by $L=\int fd\mu$ and
\[
\sigma_f =\lim_{t\rightarrow+\infty}\int_{\Lambda} \left(\int_0^{t}f(\psi^t (x))dt-t\int fd\mu \right )^2d\mu (x).
\]
Note that this theorem is more general than Theorem \ref{T1}, in the sens that we don't have necessarily $p(\tau)=\int_\tau f$ for the periodic orbits $\tau$ of $\psi^t$.  
\section{Proof of the main results}
\subsection{Proof of Theorem \ref{T1}}
\begin{proof}
Under the assumptions of Theorem \ref{T1} ($c$ Hölder continuous and $0<h_c <\infty$), one can apply the reparametrizing Theorem \ref{samb}. We have then a proper and cocompact action $\Gamma \curvearrowright \backslash \partial^2 \Gamma \times \mathbb{R}$
\[
\gamma (x,y,s)=(\gamma x, \gamma y, s-c(\gamma, y)),
\]
and a translation flow on the quotient space, $\psi^{t}: \Gamma \backslash \partial^2 \Gamma \times \mathbb{R} \circlearrowleft$
\[
\psi^{t} \Gamma(x,y,s)=\Gamma (x,y, s-t).
\]
For $\gamma \in \Gamma$ primitive, the periods $c(\gamma, \gamma^+)$ of $c$ are the periods $p(\tau)$ of the periodic orbits $\tau$ of $\psi^{t}: \Gamma \backslash \partial^2 \Gamma \times \mathbb{R} \circlearrowleft$. 
Let $s\rightarrow \tau (s)=\Gamma \cdot (\gamma_{-}, \gamma_{+},s)$ a periodic orbit of $\psi^t$, i.e. the lift of $\tau$ to $\partial^2 \Gamma \times \mathbb{R}$, is $\tau=(\gamma_{-}, \gamma_{+}, s)$, with $\gamma \in \Gamma$ primitive and $s\in \mathbb{R}$  (we denote the orbit and the lifts by the same symbol if there is no confusion to be worried about). Since $\gamma (\gamma_{-}, \gamma_{+}, s)=(\gamma_{-}, \gamma_{+}, s-c(\gamma, \gamma^+))$, we have $p(\tau)=c(\gamma, \gamma^+)$. This is a straightforward verification by observing that, for all $s\in \mathbb{R}$,
\[
\Gamma \cdot (\gamma_{-}, \gamma_{+}, s-c(\gamma, \gamma_{+}))= \Gamma \cdot \gamma (\gamma_{-}, \gamma_{+}, s)=\Gamma \cdot (\gamma_{-}, \gamma_{+}, s)=\tau (s)=\tau (s-p(\tau))
\]
\[=\Gamma \cdot (\gamma_{-}, \gamma_{+}, s-p(\tau)).
\]
The flow $\psi^{t}$ is the reparametrization of the geodesic flow by a Hölder continuous cocycle $c_F$, with $F>0$,
and we have by Theorem \ref{L},
\begin{equation}\label{éq}
p(\tau)=c(\gamma, \gamma_+)=c_F (\gamma, \gamma_+)=\int_{\tau}F.
\end{equation}
Then, 
\begin{eqnarray*}\label{EQ}
&&\frac{\#\{[\gamma] \in [\Gamma]_t^c: 
\frac{c(\gamma, \gamma_{+})-Lt}{\sigma\sqrt{t}}\in [a, b]\}}
{\#[\Gamma]_t^c}\\
&=&
\frac{\#\{\tau \ periodic, \ p(\tau)\leq t: 
\frac{\int_{\tau}F-Lt}{\sigma\sqrt{t}}\in [a, b]\}}{\#\{\tau \ periodic, \ p(\tau)\leq t\}}.
\end{eqnarray*}
To proceed further, recall the following result (the prime orbit theorem for the cocycle $c$).
\begin{thm}[Sambarino \cite{samb1}]\label{samba}
Let $c: \Gamma \times \partial \Gamma \rightarrow \mathbb{R}$ be a Hölder continuous cocycle with positive periods and finite exponential growth rate $h_c$. Then,
\[
h_cte^{-h_c t}\#[\Gamma]_t^c \longrightarrow 1, \ as \ t\rightarrow \infty.
\]
\end{thm}
Consequently, Theorem \ref{T1} is now a consequence of Theorem \ref{samba} and Theorem \ref{sc}.
\end{proof}

\subsection{Proof of Theorem \ref{A}}
\begin{proof}
We prove in this section the following result.
\begin{thm}\label{T4}
Let $\rho : \Gamma \rightarrow G$ be Zariski dense and $P$-convex representation, and fix $\varphi$ in the interior of $\mathcal{L}_P^{*}$, where $P$ is a parabolic subgroup of $G$.
There exist $L_P^{\varphi}$ and $\sigma_P^{\varphi} >0$ such that,
\[
h_{\varphi}te^{-h_{\varphi}t} 
\#  \{ [\gamma]\in [\Gamma]_t^{\varphi}: \frac{\varphi(\beta_P (\gamma, \gamma^+))-tL_P^{\varphi}}{\sigma_P^{\varphi} \sqrt{t}}\in [a,b]  \}
\rightarrow \mathcal{N}(a, b),  
\]as  $t\rightarrow \infty$.
\end{thm}
Recall the following result.
\begin{proposition}[\cite{samb3}]\label{Prop3}
Let $\rho : \Gamma \rightarrow G$ be Zariski dense and $P$-convex representation, where $P$ is a parabolic subgroup of $G$. Consider $\varphi$
in the interior of $\mathcal{L}_P^{*}$. Then,
\begin{enumerate}
\item The cocycle $\varphi \circ \beta_P$ has finite exponential growth rate.
\item The function $\varphi \circ F_P$ is cohomologus to a Hölder  continuous positive function.  
\end{enumerate}
\end{proposition}

Let $\varphi \circ F_P : T^1M \longrightarrow \mathbb{R}$ in the Liv$\check{s}$ic cohomological class of $\mathcal{C}^2$ real functions on $T^1M$ given by the Ledrappier's Theorem \ref{L} and representing the cocycle $\varphi \circ \beta_P$. Both $\varphi \circ \beta_P$ and $\varphi \circ F_P$ have the same periods. By Proposition \ref{Prop2}, the positive numbers $\varphi (\lambda_P (\rho \gamma))$  are the periods of $\varphi \circ \beta_P$ (and $\varphi \circ F_P$).

Consequently, the positive numbers $\varphi (\lambda_P (\rho \gamma))$ are the periods of an Anosov flow that is a reparametrization of the geodesic flow by $\varphi \circ F_P$ (which is cohomologus to a Hölder  continuous positive function by Proposition \ref{Prop3}). Theorem \ref{T4} is therefore a consequence of Theorem \ref{T1}.
\end{proof}
\subsection{Proof of Theorem \ref{B}}
\begin{proof}
Let $a: G \rightarrow \mathfrak{a}$ be the Cartan projection and set $a_P= \mathfrak{J}_P \circ a: G \rightarrow \mathfrak{a}_P$. 
The objective of this section is the proof of the following result (a corollary of Theorem \ref{T4}). Recall that $a: G \rightarrow \mathfrak{a}$ is the Cartan projection and $a_P= \mathfrak{J}_P \circ a: G \rightarrow \mathfrak{a}_P$. 
\begin{thm}\label{TT4}
Let $\rho : \Gamma \rightarrow G$ be Zariski dense and $P$-convex representation, and fix $\varphi$ in the interior of $\mathcal{L}_P^{*}$, where $P$ is a parabolic subgroup of $G$. Then,
\[
h_{\varphi}te^{-h_{\varphi}t} \#  \{[\gamma] \in [\Gamma]_t^{\varphi}: \frac{ \varphi(a_P((\rho\gamma)) )-tL_P^{\varphi}}{\sigma_P^{\varphi} \sqrt{t}}\in [a,b] \}
\rightarrow
\mathcal{N}(a, b),
\]as  $t\rightarrow \infty$, where $L_P^{\varphi}$ and $\sigma_P^{\varphi}$ are the constants of Theorem \ref{T4}.
\end{thm}
Recall the following important result of Benosit \cite{Be}.
\begin{proposition}[Benoist \cite{Be}]\label{be}
Let $r>0$ and $\delta >0$. Then there exists $\epsilon >0$ such that for any $g\in G$, $(r,\epsilon)$-proximal on $\mathcal{F}_P$ one has,
\[
\|a_P (g)-\lambda_P(g) +\mathcal{G}_P (g_{-}^P , g_{+}^P )\|_{\mathfrak{a}_P}\leq \delta.
\]
\end{proposition}
By Proposition \ref{be}, for every $g\in G$ $(r,\epsilon)$-proximal on $\mathcal{F}_P$ we have,
\begin{equation}\label{EQQ}
|\varphi(a_P (g))-\varphi(\lambda_P(g)) +\varphi(\mathcal{G}_P (g_{-}^P , g_{+}^P ))|\leq \delta \|\varphi\|.
\end{equation}
Set
\[
\lambda_t (\gamma) = \frac{\varphi (\lambda_P(\rho(\gamma))) -L_P^{\varphi}t}{\sigma_P^{\varphi} \sqrt{t}},\ and \
\delta_t (\gamma) =\frac{\varphi(a_P(\rho(\gamma))) -\varphi (\lambda_P(\rho(\gamma)))}{\sigma_P^{\varphi} \sqrt{t}}.
\]
By Theorem \ref{T4}, $h_{\varphi}te^{-h_{\varphi}t} 
\# \{ [\gamma]\in [\Gamma]_t^{\varphi}: \lambda_t(\gamma)\in [a,b\}]$ converges to $\mathcal{N}(a, b)$.
We have to show that,
\[
h_{\varphi}te^{-h_{\varphi}t} 
\# \{ [\gamma]\in [\Gamma]_t^{\varphi}: \lambda_t(\gamma)+\delta_t (\gamma)\in [a,b\}] \rightarrow \mathcal{N}(a, b), 
\]
as $t\rightarrow \infty$.
Choose $\delta, r$ and $\epsilon$ as in Proposition \ref{be}.
We have seen in the proof of Lemma \ref{UD}, that given a conjugacy class, we can choose a representative $\gamma$ such that $\exp [\gamma_{-}, \gamma_{+}]_\varphi >r$. 
Thus, by Proposition \ref{sam} and (\ref{EQQ}), for all, but a finite number of the $\gamma$'s, are $(r, \epsilon)$-proximal $\gamma$ and,
\[
|\varphi(a_P (g))-\varphi(\lambda_P(g))+\log r | \leq 2\delta \|\varphi\|.
\]
This implies that for $t$ sufficiently large (depending on $\delta, r$ and $\sigma_P^\varphi$) we have $|\delta_t(\gamma)|\leq \delta$, and consequently, 
\[
\# \{ [\gamma]\in [\Gamma]_t^{\varphi}: |\delta_t(\gamma)|\leq \delta\}\geq
\# \{ [\gamma]\in [\Gamma]_t^{\varphi}: \gamma \ is \ (r, \epsilon)-proximal\}.
\]
Thus by Lemma \ref{UD},
\begin{equation}\label{E4}
h_{\varphi}te^{-h_{\varphi}t} 
\# \{ [\gamma]\in [\Gamma]_t^{\varphi}: |\delta_t(\gamma)|\leq \delta\} \longrightarrow 1,
\ as \ t \rightarrow \infty.
\end{equation}
Now, by (\ref{E4}) we get,
\begin{eqnarray*}
&& \liminf_{t\rightarrow \infty}
h_{\varphi}te^{-h_{\varphi}t} 
\# \{ [\gamma]\in [\Gamma]_t^{\varphi}: \lambda_t(\gamma)+\delta_t (\gamma)\in [a,b\}] \\
&\geq & \liminf_{t\rightarrow \infty}h_{\varphi}te^{-h_{\varphi}t} 
\# \{ [\gamma]\in [\Gamma]_t^{\varphi}: \lambda_t(\gamma)
\in 
[a+\delta, b-\delta] \ and \ |\delta_t(\gamma)|
\leq \delta\}\\
&=& \liminf_{t\rightarrow \infty}h_{\varphi}te^{-h_{\varphi}t} 
\# \{ [\gamma]\in [\Gamma]_t^{\varphi}: \lambda_t(\gamma)
\in [a+\delta, b-\delta]\}\\
&=& \mathcal{N}(a+\delta,b-\delta).
\end{eqnarray*}
For the $\limsup$ we have (using (\ref{E4}))
\begin{eqnarray*}
&& \limsup _{t\rightarrow \infty}h_{\varphi}te^{-h_{\varphi}t} 
\# \{ [\gamma]\in [\Gamma]_t^{\varphi}: \lambda_t(\gamma)+\delta_t (\gamma)\in [a,b\}]\\
&=& \limsup _{t\rightarrow \infty}h_{\varphi}te^{-h_{\varphi}t} 
h_{\varphi}te^{-h_{\varphi}t}\# \{ [\gamma]\in [\Gamma]_t^{\varphi}: \lambda_t(\gamma)
\in [a+\delta, b-\delta] \ and \ |\delta_t(\gamma)|
\leq \delta\}\\
&\leq &  \limsup _{t\rightarrow \infty}h_{\varphi}te^{-h_{\varphi}t} 
h_{\varphi}te^{-h_{\varphi}t}\# \{ [\gamma]\in [\Gamma]_t^{\varphi}: \lambda_t(\gamma)
\in [a+\delta, b-\delta]\}\\
&=& \mathcal{N}(a+\delta,b-\delta).
\end{eqnarray*}
Thus, since $\delta$ is arbitrary, we finally get,
\[
\lim _{t\rightarrow \infty}h_{\varphi}te^{-h_{\varphi}t} 
\# \{ [\gamma]\in [\Gamma]_t^{\varphi}: \lambda_t(\gamma)+\delta_t (\gamma)\in [a,b\}]=\mathcal{N}(a,b).
\]
This completes the proof of Theorem \ref{TT4}.

\end{proof}

\end{document}